\documentclass[final]{amsart}
\usepackage{times}
\usepackage{cite}
\usepackage{algorithmic}
\usepackage[pdftex]{graphicx}
\usepackage{amsmath}
\usepackage{amsfonts}
\usepackage{fixltx2e}
\usepackage{stfloats}
\usepackage{url}
\usepackage{subfig}
\usepackage{multirow}
\usepackage{rotating}
\usepackage{flushend}
\usepackage{color}
\usepackage{booktabs}
\usepackage{fullpage}

\urldef{\mailsa}\path|njam031@aucklanduni.ac.nz|
\urldef{\mailsa}\path|jmue933@aucklanduni.ac.nz|
\urldef{\mailsb}\path|{lutteroth,gerald}@cs.auckland.ac.nz|
\urldef{\mailsc}\path|dneedell@cmc.edu|

\newtheorem{lemma}{Lemma}
\newtheorem{theorem}{Theorem}

\begin{document}
%
\title{Kaczmarz Algorithm with Soft Constraints for User Interface Layout}

\author{Noreen Jamil, Deanna Needell, Johannes M\"uller, Christof Lutteroth, Gerald Weber}


\begin{abstract}
The Kaczmarz method is an iterative method for solving large systems of equations that projects iterates orthogonally onto the solution space of each equation.
In contrast to direct methods such as Gaussian elimination or QR-factorization, this algorithm is efficient for problems with sparse matrices, as they appear in constraint-based user interface (UI) layout specifications.
However, the Kaczmarz method as described in the literature has its limitations:
it considers only equality constraints and does not support soft constraints, which makes it inapplicable to the UI layout problem.

In this paper we extend the Kaczmarz method for solving specifications containing soft constraints, using the prioritized IIS detection algorithm.
Furthermore, the performance and convergence of the proposed algorithms are evaluated empirically using randomly generated UI layout specifications of various sizes.
The results show that these methods offer improvements in performance over standard methods like Matlab's LINPROG, a well-known efficient linear programming solver.
\end{abstract}
\maketitle

\section{Introduction}

Linear problems are encountered in a variety of fields such as engineering, mathematics and computer science.
To solve these problems, various numerical methods have been proposed.
These methods can be classified into direct and indirect methods, the latter also known as iterative.
Direct methods intend to calculate an exact solution in a finite number of steps, whereas iterative methods start with an initial approximation and usually produce improved approximations in a theoretically infinite sequence whose limit is the exact solution~\cite{Bhatti:Numerical-Analysis}.

Many linear problems are \textit{sparse}, i.e.\ most linear coefficients in the corresponding coefficient matrix are zero so that the number of non-zero coefficients is \emph{$O(n)$} with $n$ being the number of variables~\cite{Kunis:Sparse-Trigonometric-Polynomials}.  Sparse problems frequently occur in the domain of user interface (UI) layout, the main focus of this paper (we discuss this domain in detail in Section~\ref{UILayout}).
Since it is useful to have efficient solving methods specifically for sparse linear systems, much attention has been paid to iterative methods, which are often preferable for such cases~\cite{Anita:Numerical-MethodsforScientistandEngineers}.
The advantage is that iterative methods spend minimal processing time on coefficients that are zero.
Direct methods, on the other hand, usually lead to fill-in, i.e.\ coefficients change from an initial zero to a non-zero value during the execution of the algorithm.
In these methods we therefore lose the sparsity property and have to deal with a lot more coefficients, which makes processing slower.
Although there are some techniques to minimize fill-in effects, iterative methods are often faster than direct methods for large, sparse problems~\cite{Michele:Preconditioning}.

A common iterative method used to solve sparse linear systems is the Kaczmarz algorithm~\cite{Kaczmarz:Lettres}.
Starting with an initial guess, it selects a row index of the matrix and projects the current iterate onto the solution space of that equation, refining the solution until a sufficient precision is reached.
In our previous work~\cite{Jamil:Linear-Relaxation} we proposed extensions to the linear relaxation method to deal with over-determination.
The original linear relaxation method could not be used in such cases as it requires the problem matrix to be square.
However, the Kaczmarz method does not have this limitation, and hence seems an obvious choice for problems with non-square matrices as they occur in UI layout.

Despite its efficiency for sparse systems, the Kaczmarz method is currently not used for UI layout.
The reasons for this are two-fold.
First, the UI layout contains linear equality and \textit{inequality} constraints for specifying relationships among objects such as ``inside", ``above", ``below", ``left-of", ``right-of" and ``overlap".
Although the Kaczmarz algorithm and its variants are not designed to handle inequality constraints, preliminary work on the Kaczmarz method for inequality constraints suggests the natural adaptation, which ignores inequality constraints if they are already satisfied and otherwise treats them as equations~\cite{LL10:Randomized-Methods}.
We also adapt this heuristic approach for the UI problem.

The second issue that we face in UI layout and many other problems is that the system may contain \textit{conflicting} constraints.
This may happen by over-constraining, i.e.\ by adding too many constraints, making the system infeasible.
If a specification contains conflicting constraints, the basic Kaczmarz method will not converge.
To resolve conflicts, the notion of \emph{soft constraints} can be introduced.
In contrast to the usual \emph{hard} constraints, which cannot be violated, soft constraints may be violated as much as necessary if no other solution can be found.
Soft constraints can be prioritized so that in a conflict between two soft constraints only the soft constraint with the lower priority is violated.
This leads naturally to the notion of \emph{constraint hierarchies}, where all constraints are essentially soft constraints, and the constraints that are considered ``hard'' simply have the highest priorities~\cite{Wilson:Constraint-hierarchies}.
Using only soft constraints has the advantage that a problem is always solvable, which cannot be guaranteed if hard constraints are used.

We propose a conflict resolution algorithm for solving systems of prioritized linear constraints with the Kaczmarz method.
In the algorithm, non-conflicting constraints are successively added in descending order of priority.
This algorithm yields conflict-free subproblems to a given problem.
There are already algorithms for finding feasible subsystems, but they differ from our approaches as they do not take into account prioritized constraints~\cite{John:Fast-Heuristics}.

With the presented conflict resolution algorithm, Kaczmarz can be applied to overdetermined linear constraint problems, for example in the domain of UI layout.
This was experimentally evaluated with regard to convergence and performance, using randomly generated UI layout specifications.
The results show that the proposed algorithm is optimal and efficient.
Furthermore, we observe that our implementation outperforms Matlab's LINPROG linear optimization package~\cite{linprog:solver}, LP-Solve~\cite{lp:solve} and the implementation of QR-decomposition of the Apache Commons Math Library~\cite{Commons:Math}.
LP-Solve is a well-known linear programming solver that has been used for UI layout.
The implementation of QR-decomposition of the Apache Commons Math Library is an example of a direct method.

\subsection{Organization}

The remainder of the paper is organized as follows.
We begin with a short overview of constraint-based UIs and iterative solution procedures for linear systems in Section~\ref{sec:back}.
We discuss related work in Section~\ref{RelatedWork}.
In Section~\ref{KaczmarzMethod} we describe the Kaczmarz method in detail, and explain how support for inequalities and soft constraints was added.
The methodology and the results of the experimental evaluation are presented in Section~\ref{Evaluation}.
Section~\ref{Conclusion} summarizes conclusions and provides an outlook on future work.

\section{Background}\label{sec:back}

In~\cite{Jamil:Linear-Relaxation} we extended linear relaxation to handle over-determined systems and conflicting constraints
for UI layout.
We proposed two pivot assignment algorithms that can be used with any problem matrix, regardless of its shape or diagonal elements.
The first algorithm selects pivot elements pseudo-randomly and the second algorithm selects pivot elements according to certain criteria.
However, computing the pivot assignment for solving non-square matrices can make linear relaxation slow.
To overcome this, we consider the Kaczmarz method, which does not require a pivot assignment.

\subsection{User Interface Layout as a Linear Problem} \label{UILayout}
Constraints are a suitable mechanism for specifying the relationships among objects.
They are used in the area of logic programming, artificial intelligence and UI specification.
They can be used to describe problems that are difficult to solve, conveniently decoupling the description of the problems from their solution.
Due to this property, constraints are a common way of specifying UI layouts, where the objects are widgets and the relationships between them are spatial relationships such as alignment and proportions.
In addition to the relationships to other widgets, each widget has its own set of constraints describing properties such as minimum, maximum and preferred size.

UI layouts are often specified with linear constraints~\cite{Weber:High-Level-Constraints}.
The positions and sizes of the widgets in a layout translate to variables.
Constraints about alignment and proportions translate to linear equations, and constraints about minimum and maximum sizes translate to linear inequalities.
Many of the constraints are soft because they describe desirable properties in the layout (e.g.\ preferred sizes), which cannot be satisfied under all conditions (e.g.\ all layout sizes).
Furthermore, the resulting systems of linear constraints are sparse.
There are constraints for each widget that relate each of its four boundaries to another part of the layout, or specify boundary values for the widget's size, as shown in Figure~\ref{fig:Gui-diagram}.
As a result, the direct interaction between constraints is limited by the topology of a layout, resulting in sparsity.

\begin{figure}[t]
\centering
\includegraphics[width=0.6\columnwidth]{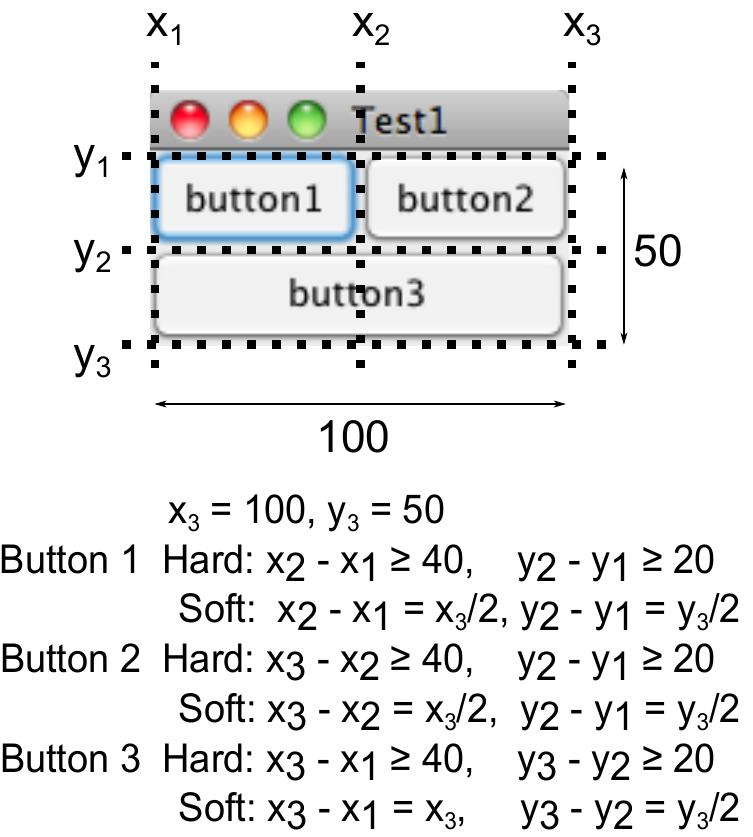}
\caption{An example of constraint-based UI layout}\label{fig:Gui-diagram}
\end{figure}

Because the UI layout problem contains inequality and soft constraints, existing UI layout solvers use algorithms other than Kaczmarz.
Some of these solvers will be discussed in the following section.

\section{Related Work} \label{RelatedWork}

Different direct and iterative methods exist which can solve least-squares problems.
Examples of direct methods are QR-factorization~\cite{Biswa:N-Linear-Algebra} and normal equations~\cite{Michael:Normal-Equations}.
Iterative methods that are used to solve linear systems include simplex~\cite{Dantzig:Linear-Programming}, revised simplex~\cite{Hamdy:Operations-Research}, conjugate gradient~\cite{Hestenes:1952uq}, generalized minimal residual~\cite{Saad:1986fk}, and others~\cite{axelsson1996iterative}.
In order to solve over-determined systems in a least-squares sense, Herman adds some additional constraints and variables, similar to slack variables in the simplex approach~\cite{Herman:Iterative}.
The resulting method converges towards the least-squares solution.
Popa analyzed a similar approach for solving least-squares problems~\cite{Popa:Least,Rafal:Least}.
Censor showed that if the relaxation parameter goes to zero then the Kaczmarz method converges to a weighted least-squares solution for inconsistent systems~\cite{Censor:underrelaxation}.

Most of the research related to GUI layout involves various algorithms for solving constraint hierarchies.
Research related to constraint-based UI layout has provided results in the form of tools~\cite{Hosobe:Simplex-Based, Hosobe:scalable-linear-constraint} and algorithms~\cite{Badros:cassowary,Stuckey:arithmetic-constraints} for specific tasks.
Our work is concerned with two different aspects.
We must find a solution for linear inequality constraints with iterative methods, while also being able to handle soft constraints.
We will discuss related work for both aspects in turn.

Various algorithms were proposed for solving linear inequality constraints in UI layout.
The Indigo algorithm~\cite{Alan:Indigo} uses interval propagation to solve acyclic collections of inequality constraints, however, it does not handle simultaneous equality and inequality constraints.
This is overcome by the Detail method~\cite{Hosobe:Generalized-local-propagation, Hosobe:constraint-simultaneous}, which can solve linear equality and inequality constraints simultaneously.

The Cassowary solver~\cite{Badros:cassowary} can also handle linear inequalities.
It uses the simplex algorithm, and inequalities are solved by introducing slack variables.
QOCA~\cite{Finlay:constraint-solving-toolkit} intends to overcome the difficulties in maximizing the efficiency and facilitating the re-use of the solver in other applications.
This solver introduces slack variables to convert inequality constraints into equality constraints in a similar way to the Cassowary solver.
The HiRise constraint solver~\cite{Hosobe:scalable-linear-constraint} resolves both equality and inequality constraints in combination with quasi-linear optimization.

All constraint solvers for UI layout must support over-determined systems.
The commonly used techniques for dealing with over-determined problems are weighted constraints and constraint hierarchies~\cite{Meseguer:over-constrained,Matsuoka:constraint-hierarchies}.
Weighted constraints are typically used with some general forms of direct methods, while constraint hierarchies are especially utilized in linear programming based algorithms.
Many UI layout solvers are based on linear programming and support soft constraints using slack variables in the objective function~\cite{Badros:cassowary,Stuckey:arithmetic-constraints,Finlay:constraint-solving-toolkit,Weber:High-Level-Constraints}.

Most of the direct methods for soft constraint problems are least-squares methods such as LU-decompos\-ition and QR-decompo\-sition~\cite{Yasuhiro:least-squares}.
The UI layout solver HiRise~\cite{Hosobe:scalable-linear-constraint} is an example of this category.
HiRise2~\cite{Hosobe:Simplex-Based} is an extended version of the HiRise constraint solver which solves hierarchies of linear constraints by applying an LU-decompos\-ition-based simplex method.

Many different local propagation algorithms have been proposed for solving constraint hierarchies in UI layout.
The DeltaBlue~\cite{Freeman:Delta-Blue}, SkyBlue~\cite{Skyblue:local-propagation-constraint-solver} and Detail~\cite{Hosobe:constraint-simultaneous} algorithms are examples of this category.

The problem of finding the largest possible subset of constraints that has a feasible solution given a set of linear constraints is widely known as the maximum feasible subsystem (MaxFS) problem~\cite{John:Fast-Heuristics}.
The dual problem to this is the problem of finding the irreducible infeasible subsystem (IIS)~\cite{Amaldi:Maximum-Feasible}.
If one more constraint is removed from an IIS, the subsystem will become feasible.
For both problems different solving methods are proposed.

There are non-deterministic and deterministic methods to solve the MaxFS problem.
Some of these methods use heuristics~\cite{Amaldi: Two-Phase, O:Feasible-System}, but only a few methods solve the problem deterministically.
The branch and cut method proposed by Pfetsch~\cite{Pfetsch:Branch-Cut} is an example of a deterministic method.

Besides methods for MaxFS there are also some methods to solve the IIS problem.
These methods are: deletion filtering, IIS detection and grouping constraints.
Deletion filtering~\cite{Dravnieks:Minimal} removes constraints from the set of constraints and checks the feasibility of the reduced set.
IIS detection~\cite{Tamiz:IIS} starts with a single constraint and adds constraints successively.
The grouping constraints method~\cite{Guieu:MixedInteger} was introduced to speed up the aforementioned algorithms by adding or removing groups of constraints simultaneously.
Even though these methods deal with the problem of finding a feasible subsystem, it is not possible to apply them directly.
The main reason is that they do not consider prioritized constraints as we do in our approach.

\section{Kaczmarz Method} \label{KaczmarzMethod}

The Kaczmarz method is an iterative method used for solving large-scale over-determined linear systems of equations~\cite{Kaczmarz:Lettres}.  It is also used in tomography, and in that setting is called the ``algebraic reconstruction technique'' (ART)~\cite{GBH70:Algebraic-Reconstruction}.
Given a system of \emph{m} equations and \emph{n}  variables of the form
\begin{equation}\label{linearequation}
    Ax=b,
\end{equation}
the Kaczmarz method projects orthogonally onto the solution hyperplane of each constraint in the system sequentially.
The algorithm can thus be described as follows.
\begin{equation}\label{MainKaczmarz}
x_{k+1}=x_{k} + \omega \frac{(b_i - a_{i} \cdot x_k) a_{i}}{\|a_i\|^2}
\end{equation}
where $x_{k}$ is the $k$-th iterate, $i = (k\mod m) +1$ (for deterministic Kaczmarz), $a_i$ is the $i$-th row of the matrix $A$, $b_i$ is the $i$-th component of the right-hand side vector, and $\|a\|$ denotes the Euclidean norm of the vector $a$.
Alternatively, to randomize the Kaczmarz method, we can choose a random $i$ with $1 \leq i \leq m$ for each $k$.

Starting with an initial estimate $x_0$, the method projects the current iterate onto the solution space of the next equation.
The algorithm iterates until the relative approximate error is less than a pre-specified tolerance. $\omega$ is an optional relaxation parameter that is set to 1 in the original Kaczmarz method.

\subsection{Convergence}\label{sec:Relaxation:Convergence}

The  Kaczmarz method is guaranteed to converge if $\omega$ lies inside the interval $(0,2)$~\cite{Kaczmarz:Lettres,N-TMCT1986}. 
In this section we give the convergence proof using the terminology explained below.

\begin{lemma}[Translation invariance]
Let the Kaczmarz method for \(Ax=b\) converge to \(\bar x\) starting with $x_{0}$. Then the Kaczmarz method for the homogeneous system \(Ay=0\) starting with $y_{0} =x_{0}  - \bar x$ will have the same convergence behavior, i.e.  $y_{k} =x_{k}  - \bar x$ for all $k$.
\end{lemma}
The proof is by induction. The induction step follows trivially from the linear definition of the iteration step.

\begin{lemma}[Convergence of homogeneous system]
The Kaczmarz method for the homogeneous system \(Ay=0\) with nonsingular $A$ converges exponentially for every initial guess~$y_{0}$.
\end{lemma}

\begin{proof}
The Kaczmarz method is a linear method; by definition, the change to the estimate in every iteration step is a linear function that can be modeled with an iteration matrix  $K_i(A)$.
We show that the spectral radius $\rho$ fulfills $\rho(K_i(A))=1.$
It suffices to show that all vectors are transformed to vectors of shorter or equal size
(if $\rho(K_i(A)) > 1$ there would be a vector that gets longer).
We observe that each solution hyperplane of the homogeneous system goes through the origin.
Since the iteration step performs an orthogonal projection, for $\omega =1$ the origin and the points
 $y_{k}$ and $y_{k+1}$ form a triangle with a right angle at $y_{k+1}$.
Hence
\begin{equation}\label{eq:convergence}
\|y_{k+1}\| \leq \|y_{k}\|
\end{equation}
by Pythagoras.
For $0 < \omega <2$, $y_{k+1}$ is equal in length to a weighted vector sum of  $y_{k}$ and the result for $\omega =1$, so its length is intermediate and~\eqref{eq:convergence} still holds.
Hence we have shown that $\rho(K_i(A)) \leq 1$.
But if the estimate is already a solution for the constraint, then it does not move, hence giving $\rho(K_i(A)) = 1$.
We now look at the product
\begin{equation}
K(A) = \prod_{1 \leq i \leq m} K_i(A).
\end{equation}
For any $y_0$  we have 
\begin{equation}
K(A) y_0 = y_m.
\end{equation}
We now show for any $y_0 \neq 0$ that $\|y_0\| > \|y_m\|$.
There must be one $i$ so that $y_{i-1} \neq y_i$, since $A$ is nonsingular and hence $y_0 \neq 0$ cannot fulfill all constraints at once.
Since in step $i$ we have  $y_{i-1} \neq y_i$, we also have by Pythagoras $\|y_{i-1}\| > \|y_i\|$, as explained above.
In all other steps the error does not increase, hence we know
\begin{equation}
\|K(A) z\| < \|z\|.
\end{equation}
This means, overall we know $\rho(K(A)) = c < 1$.
Hence the error of the estimate decreases over the course of $m$ iterations by at least $c$,
and overall we get an exponential convergence behavior with base $\leq \sqrt[m]{c}$.
\end{proof}

The proof can be easily generalized to a singular $A$ by only considering the nonsingular orthogonal subspace.
Both lemmata together clearly give:

\begin{theorem}
The Kaczmarz method for \(Ax=b\) converges for every initial guess $x_{0}$.
\end{theorem}

The constant $c$ is a characteristic of the problem matrix $A$, similar to the condition number.
Since the convergence rate of the Kaczmarz method depends on $c$, it is imaginable that pre-conditioners could be used to reduce $c$ and enhance the convergence speed.
Pre-conditioners are used for other iterative methods such as linear relaxation, e.g.\ scaling algorithms~\cite{Daniel:Scaling} and bipartite matching algorithms~\cite{Duff:Permuting}. 
Such algorithms may scale the infinity norm of both rows and columns in a matrix to 1 or permute large entries to the diagonal of a sparse matrix.
We usually have well-conditioned coefficient matrices for UI layout problems, for which the Kaczmarz method converges quickly.

As described by~\eqref{MainKaczmarz}, the convergence rate of the Kaczmarz method may depend on the ordering of the rows of the matrix $A$.
A problematic ordering can lead to a drastically reduced rate of convergence.
To overcome this, a randomized variant can be used.
Strohmer and Vershynin proposed a further variant with weighted probabilities proportional to the norm of the $i$th row~\cite{SV09:Randomized-Kaczmarz}.

In the inconsistent case, it has been shown that the method exhibits the same convergence down to a threshold~\cite{Nee10:Randomized-Kaczmarz}, and modified methods even converge to the least-squares solution~\cite{censor1981row,galantai2003projectors,ZF12:Randomized-Extended}.
The convergence rate can further be improved by selecting blocks of rows at a time for the projection~\cite{Elf80:Block-Iterative-Methods,EHL81:Iterative-Algorithms,NW12:Two-Subspace-Projection,needell2013paved}.

\subsection{Inequalities} \label{sec:Inequalities}

The Kaczmarz method supports only linear equations, but we extend this algorithm for solving linear inequalities in a natural way, as in~\cite{LL10:Randomized-Methods}.
In each iteration, the algorithm ignores inequalities if they are satisfied, and otherwise treats them as if they were equations.
This means that inequalities influence the solving process only if this is necessary.

\begin{figure}[t]
\centering
\includegraphics[width=\columnwidth, height=2.6in]{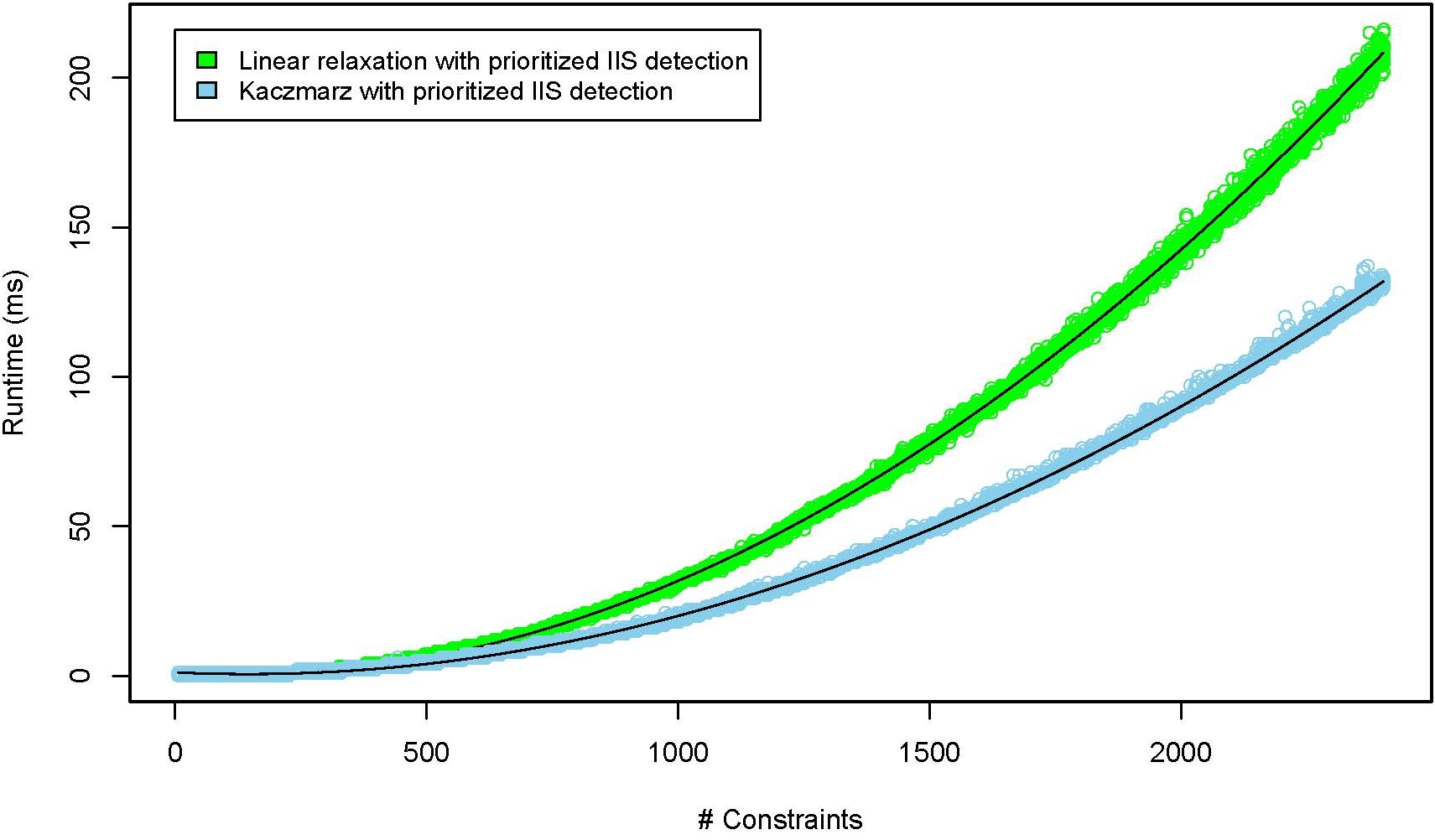}
\caption{Performance comparison of linear relaxation and Kaczmarz.} \label{fig:random}
\end{figure}

\subsection{Soft Constraints} \label{SoftConstraints}

For many problems, including UI layout, conflicting constraints occur naturally in specifications, as they express properties of a solution that are desirable but not mandatory.
As a result, soft constraints need to be supported, which are satisfied if possible, but do not render the specification infeasible if they are not.
A natural way to support soft constraints is to treat all constraints as soft constraints, with different priorities.
These priorities can be defined as a total order on all constraints that specifies which one of two constraints should be violated in case of a conflict.

To define the solution of a system of prioritized soft constraints we first have to define the subset $E$ of constraints which we call \emph{enabled constraints}.
We consider the characteristic function $\mathbf{1}_E : \mbox{\it Constraints} \to \{ 0,1\}$ of $E$, which expresses whether a constraint is contained in $E$, to construct an integer in binary representation (\(\iota\)).
According to their priority, each constraint is represented by a bit of that integer, with constraints of higher priority taking the more significant bits.
Then such subsets can be compared by using the numerical order $\geq$ of the integers.
We are interested in the subset that is largest in that order and still fulfills the following property:
all constraints in the subset are non-conflicting.

To add support for soft constraints to the Kaczmarz method, we use the prioritized IIS detection algorithm~\cite{Jamil:Linear-Relaxation}, which approximates the maximum characteristic function starting from the most significant bit.
This algorithm starts with an empty set $E$ of enabled constraints.
It then adds constraints incrementally in order of descending priority so that $E$ is conflict-free, until all non-conflicting constraints have been added.
Iterating through the constraints, the algorithm adds each constraint tentatively to $E$ (``enabling'' it), and tries to solve the resulting specification.
If a solution is found, the constraint is kept.
Otherwise, the added constraint is removed again, in which case the previous solution is restored.
Finally, the algorithm proceeds to the constraint with the next lower priority, until all constraints have been considered.
This algorithm assumes that the method used for solving the system converges if there is no conflict, which is the case for Kaczmarz.

\section{Experimental Evaluation} \label{Evaluation}

In this section we present an experimental evaluation of the proposed algorithm.
We conduct two different experiments to evaluate (i) the convergence behavior, and (ii) the performance in terms of computation time.
The experiments are described in the following.

\begin{figure}[t]
\centering
\includegraphics[width=\columnwidth, height=2.6in]{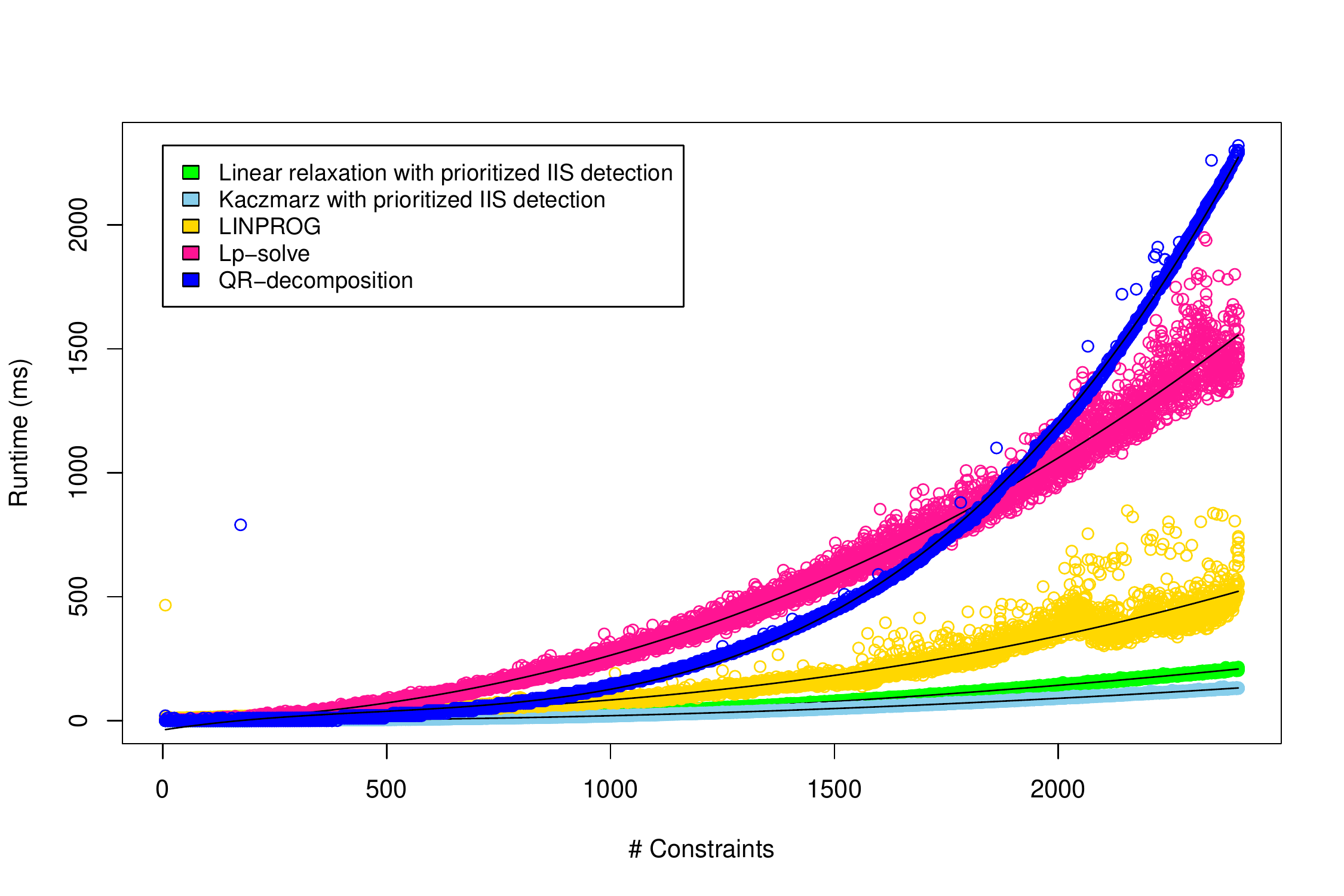}
\caption{Performance comparison of linear relaxation with prioritized IIS detection and random pivot assignment, Kaczmarz with prioritized IIS detection, LINPROG, LP-Solve, and QR-decomposition}
\label{fig:lp-solve}
\end{figure}

\subsection{Methodology}

For all experiments we used the same hardware and test data generator, but instrumentalized the algorithms differently.
We used the following setup:
a desktop computer with Intel i5 3.3GHz processor and 64-bit Windows 7, running an Oracle Java virtual machine.
Layout specifications were randomly generated using the test data generator described in~\cite{Weber:High-Level-Constraints}.
For each experiment the same set of test data was used.
The specification size was varied from 4 to 2402 constraints, in increments of 4 constraints (2 new constraints for positioning and 2 new constraint for the preferred size of a new widget).
For each size 10 different layouts were generated, resulting in a total of 6000 different layout specifications.
A tolerance of 0.01 was used for solving.
For Kaczmarz a relaxation parameter of 1.0 was used; for linear relaxation a slightly smaller relaxation parameter of 0.7 had to be used to avoid problems of divergence.

In the first experiment we investigated the convergence behavior of the algorithms.
We measured for each algorithm the number of sub-optimal solutions.
A solution is sub-optimal if the error of a constraint (the difference between the right-hand and left-hand side) is bigger than the given tolerance.

In the second experiment we measured the performance in terms of computation time (\(T\)) in milliseconds (ms), depending on the problem size measured in number of constraints~($c$).
The proposed algorithm was used to solve each of the problems of the test data set and the time was measured.
As a reference, all the generated specifications were also solved with an implementation of linear relaxation with prioritized IIS detection~\cite{Jamil:Linear-Relaxation}, Matlab's LINPROG solver~\cite{linprog:solver} and LP-Solve~\cite{lp:solve}.
We selected these solvers as linear relaxation is another standard iterative method similar to Kaczmarz, LINPROG is widely known for its speed, and LP-Solve was previously used to solve UI layout problems~\cite{Weber:High-Level-Constraints}.
Additionally, we wanted to compare our algorithm with a direct method, so we also included the implementation of QR-decomposition in the Apache Commons Mathematics Library~\cite{Commons:Math} in the evaluation.

\begin{table}[t]
\begin{center}
\begin{tabular}{rl}
  \toprule
 Symbol & Explanation \\
  \midrule
 \(\beta_0\) & Intercept of the regression model\\
 \(\beta_{1-3}\) & Estimated model parameters\\
 \(c\) & Number of constraints\\
  \(T\) & Measured time in milliseconds\\
   \(R^2\) & Coefficient of determination of the regression models\\
\bottomrule
\end{tabular}
\caption{Symbol table}
\label{tab:symbs}
\end{center}
\end{table}

\newcommand{\sign}[1]{\textsuperscript{#1}}
\begin{table*}[t]
\begin{center}
\begin{tabular}{rlllll}
  \toprule
\multicolumn{1}{c}{\bf Strategy}                       &\multicolumn{1}{c}{\(\beta_0\)} &\multicolumn{1}{c}{\(\beta_1\)} & \multicolumn{1}{c}{\(\beta_2\)}  & \multicolumn{1}{c}{\(\beta_3\)} & \(R^2\)  \\
  \midrule
Kaczmarz (prioritized IIS detection) &\(\;\;\,1.136 \)\sign{***} & \(-7.740 \cdot 10^{+03}\)\sign{***} & \(\;\;\, 2.723 \cdot 10^{-05}\)\sign{***} & \(-5.587 \cdot 10^{-10}\)\sign{***}&\(0.9994\) \\
Linear relaxation (prioritized IIS detection) &\(\;\;\, 1.035 \)\sign{***} & \(-1.112 \cdot 10^{-02}\)\sign{***} & \(\;\;\, 4.278 \cdot 10^{-05}\)\sign{***} & \(-9.176 \cdot 10^{-10}\)\sign{***}&\(0.9994\) \\
LINPROG &\(\;\;\, 18.29\)\sign{***} & \(\;\;\, 1.591 \cdot 10^{-04}\) & \(\;\;\, 4.934 \cdot 10^{-05}\)\sign{***} & \(\;\;\, 1.577 \cdot10^{-08}\)\sign{***} &\(0.9367\) \\
LP-Solve &\( -2.491\)\sign{***} & \(\;\;\, 3.924 \cdot 10^{-02}\)\sign{***} & \(\;\;\, 2.079 \cdot 10^{-04}\)\sign{***} & \(\;\;\, 1.904 \cdot 10^{-08}\)\sign{***} &\(0.9900\) \\
QR-Decomposition  &\( -37.70\)\sign{***} & \(\;\;\, 0.2802 \)\sign{***} & \(-4.009 \cdot 10^{-04}\)\sign{***} & \(\;\;\, 2.850 \cdot 10^{-07}\)\sign{***} & \(0.9989\) \\
\bottomrule
\multicolumn{6}{l}{Significance codes: \sign{***} \(p<0.001\)} \\
\end{tabular}
\caption{Regression models for the different solving strategies}
\label{tab:reg1}
\end{center}
\end{table*}

\subsection{Results}

The first experiment tested the convergence behavior of the algorithms.
We found that all algorithms converge, which is expected since the algorithms were designed to find a solvable subproblem.

In the second experiment we investigated the performance behavior of the algorithms.
To identify the performance trend of the algorithms over $c$, we defined some regression models (linear, quadratic, log, cubic).
We found that the best-fitting model is the polynomial model
\[
T(c) = \beta_0 + \beta_1c + \beta_2c^2 + \beta_3c^3 + \epsilon,
\]
which gave us a good fit for the performance data.
Table~\ref{tab:symbs} explains the symbols used.
Key parameters of the models are depicted in Table~\ref{tab:reg1}; a graphical representation of the models can be found in Figures~\ref{fig:random}~and~\ref{fig:lp-solve}.

Figure~\ref{fig:random} illustrates the performance of Kaczmarz with prioritized IIS detection and linear relaxation with prioritized IIS detection.
As the graphs show, Kaczmarz with prioritized IIS detection exhibits a better performance than linear relaxation with prioritized IIS detection.
Figure~\ref{fig:lp-solve} compares the two aforementioned algorithms to LINPROG, LP-Solve and QR-decomposition.
Kaczmarz with prioritized IIS detection performs significantly better than LINPROG, LP-Solve and QR-decomposition, especially for bigger problems.

\subsection{Discussion}

The performance results show that the Kaczmarz method with prioritized IIS detection is the fastest, and the direct method, QR-decomposition, is the slowest for UI layout problems.
The purely iterative algorithms, Kaczmarz and linear relaxation, are both faster than QR-decomposition, LP-Solve and LINPROG.
A plausible reason why LINPROG and LP-Solve are slower is that they use the simplex algorithm with one direct method solving step per iteration.
As described earlier, direct methods suffer from fill-in effects when solving sparse systems, which is generally a disadvantage compared to iterative methods in this case.

The Kaczmarz algorithm with prioritized IIS detection exhibits a better performance than linear relaxation with prioritized IIS detection.
A likely factor contributing to this is that for Kaczmarz a slightly larger relaxation parameter than for linear relaxation can be used.
Smaller relaxation parameters may slow down the convergence of iterative methods, potentially requiring more iterations to converge.
However, it was not possible to increase the relaxation parameter of linear relaxation as this would cause divergence for some of the problems.

The runtime of the two linear programming solvers exhibits a much larger variance compared to the purely iterative solvers.
One possible reason for this is that for some cases the direct methods used in the linear programming solvers are particularly inefficient, e.g.\ due to fill-in effects.
A smaller variance and hence more predictable runtime is particularly beneficial for the UI layout domain because large changes in runtime can affect the user experience, e.g.\ when resizing a GUI window interactively.

\section{Conclusion} \label{Conclusion}

We have proposed a new algorithm for using Kaczmarz for solving constraint-based UI layout problems.
In particular, we presented the following contributions:
\begin{itemize}
  \item Extensions of Kaczmarz for solving linear inequality and prioritized soft constraints.
	\item An experimental evaluation demonstrating the feasibility of using Kaczmarz for UI layout.
	\item Experimental data indicating that Kaczmarz outperforms modern linear programming solvers, linear relaxation and a recent implementation of QR-decomposition.
\end{itemize}

With the contributions mentioned above, we have demonstrated that the Kaczmarz method can be used efficiently in solvers for constraint-based UIs.
As a future work, we will investigate the effect of the relaxation parameter $\omega$ on the convergence behavior for UI layout problems.
Furthermore, we will investigate the use of pre-conditioners.

\bibliographystyle{IEEEtran}
\bibliography{bibliography}
\end{document}